\documentclass[11pt, a4paper]{article}

\title{Quasiconformal deformation of the chordal Loewner driving function and first variation of the Loewner energy}
\date{March 5, 2024}

\usepackage[T1]{fontenc}

\usepackage[latin9]{inputenc}

\usepackage{mathtools}
\usepackage{lmodern}
\usepackage{amsthm,amsmath,amsfonts,amssymb}
\usepackage{graphicx}
\usepackage{enumerate,enumitem}
\usepackage{authblk}
\usepackage{cite,url}
\usepackage[normalem]{ulem}
\usepackage{calrsfs}

\usepackage{multirow}
\usepackage{array}
\newcolumntype{P}[1]{>{\centering\arraybackslash}p{#1}}
\usepackage{longtable,booktabs}

\usepackage{hyperref}
\hypersetup{
    colorlinks=true, 
    linkcolor=black, 
    urlcolor=black, 
    citecolor=black,
    linktoc=all 
}

\usepackage[font=normal,labelfont=bf,labelsep=quad]{caption}

\usepackage[activate={true,nocompatibility},final,tracking=true,kerning=true,spacing=true,factor=1100,stretch=20,shrink=20]{microtype}
\microtypecontext{spacing=nonfrench}

\allowdisplaybreaks

\setlist[enumerate]{topsep = 1ex, leftmargin=1cm, itemsep= -2pt}
\setlist[itemize]{topsep = 1ex, leftmargin=1cm, itemsep= -2pt}

\let\OLDthebibliography\thebibliography
\renewcommand\thebibliography[1]{
  \OLDthebibliography{#1}
  \setlength{\parskip}{1pt}
  \setlength{\itemsep}{2pt}
}

\newtheorem{thm}{Theorem}[section]
\newtheorem{cor}[thm]{Corollary}

\newtheorem{lem}[thm]{Lemma}
\newtheorem{prop}[thm]{Proposition}

\theoremstyle{definition} 
\newtheorem{df}[thm]{Definition}

\newtheorem{remark}[thm]{Remark}

\newcommand{\inversion}{\iota}
\newcommand{\emu}{{\vare \mu}}
\newcommand{\enu}{{\vare \nu}}

\newcommand{\Schwarzian}{\mathcal{S}}
\newcommand{\Nonlinearity}{\mathcal{N}}

\numberwithin{equation}{section}

\usepackage{floatrow}

\usepackage{mathtools}

\usepackage[dvipsnames]{xcolor}

\global\long\def\ii{\mathfrak{i}}

\newcommand{\abs}[1]{\left\lvert #1 \right \rvert}

\newcommand{\norm}[1]{\lVert #1 \rVert}

\newcommand{\mc}[1]{\mathcal{#1}}
\newcommand{\m}[1]{\mathbb{#1}}

\def\ie{i.e., }

\renewcommand\Re{\operatorname{Re}}
\renewcommand\Im{\operatorname{Im}}

\def\SLE{\operatorname{SLE}}

\def\g{\gamma}

\def\l{\lambda}

\def\o{\omega}
\def\O{\Omega}
\def\vare{\varepsilon}

\def\Chat{\hat{\m{C}}}

\def\dd{\mathrm{d}}

\def\1{\mathbf{1}}

 \def\domain{D}

\author{Jinwoo Sung\thanks{\protect\url{jsung@math.uchicago.edu} University of Chicago, Chicago, IL, USA },\qquad Yilin Wang\thanks{\protect\url{yilin@ihes.fr} Institut des Hautes \'Etudes Scientifiques, Bures-sur-Yvette, France
}}

\begin{document}

\maketitle
\begin{abstract}
We derive the variational formula of the Loewner driving function of a simple chord under infinitesimal quasiconformal deformations with Beltrami coefficients supported away from the chord. 
As an application, we obtain the first variation of the Loewner energy of a Jordan curve, defined as the Dirichlet energy of its driving function.
This result gives another explanation of the identity between the Loewner energy and the universal Liouville action introduced by Takhtajan and Teo, which has the same variational formula. We also deduce the variation of the total mass of SLE$_{8/3}$ loops touching the Jordan curve under quasiconformal deformations.
\end{abstract}



\bigskip

{\bf Acknowledgments:}
We thank Gregory Lawler, Catherine Wolfram, and the anonymous referee for useful suggestions. Part of the work on this paper was carried out during visits by J.S.\ to l'Institut des Hautes \'Etudes Scientifiques, whose hospitality is gratefully acknowledged. J.S. was partially supported by a Kwanjeong Educational Foundation Fellowship. 
Y.W. is funded by the European Union (ERC, RaConTeich, 101116694). Views and opinions expressed are however those of the author(s) only and do not necessarily reflect those of the European Union or the European Research Council Executive Agency. Neither the European Union nor the granting authority can be held responsible for them.

\bigskip

\section{Introduction} 

   One hundred and one years ago, Loewner introduced \cite{Loewner1923} a method to encode a simple planar curve by a family of uniformizing maps (called the Loewner chain) which satisfies a differential equation driven by a real-valued function.
 This method has become a powerful tool in geometric function theory. It was instrumental in the proof of Bieberbach conjecture by De Branges \cite{DeBranges1985} (which was also the original motivation of Loewner) and was revived around 2000 as a fundamental building block in the definition of the Schramm--Loewner Evolution \cite{schramm2000scaling}.
   On the other hand, quasiconformal mapping is one of the fundamental concepts in geometric function theory and Teichm\"uller theory. Thus, we find it natural to investigate the interplay between quasiconformal maps and the Loewner transform.   We will further comment on the motivation of this work and discuss follow-up questions in Section~\ref{sec:comments}.
   We mention that analytic properties of the Loewner driving function have been investigated in, e.g., \cite{Rohde_Schramm,Marshall_Rohde,Lind_sharp, Lind_Tran,FrizShekhar15,RW}.

Our first result shows how quasiconformal deformations of the ambient domain  $\m H = \{z \in \mathbb C \colon \Im (z) > 0\}$ affect the driving function of a simple chord in $\m H$ connecting $0$ to $\infty$.

\begin{thm}\label{thm:intro_driving_function} Let $\eta$ be a simple chord in $(\m H; 0, \infty)$ under capacity parametrization and $\nu \in L^\infty (\m H)$ be an infinitesimal Beltrami differential whose support is compact and disjoint from $\eta$. For $\vare \in \m R$ such that $\norm{\vare\nu}_{\infty} <1$, let $\psi^\enu$ be the unique quasiconformal self-map of $\m H$ with Beltrami coefficient $\enu$ such that $\psi^\enu(0) = 0$ and $\psi^\enu(z) - z = O(1)$ as $z\to\infty$. Denote the capacity and driving functions of the parametrized chord $\psi^{\enu} \circ \eta$ in $(\m{H},0,\infty)$ by $a_\cdot^\enu$
    and $\lambda_\cdot^\enu$, respectively. Then, 
    \begin{equation}
        \frac{\partial \lambda_t^\enu}{\partial \vare} \bigg|_{\vare=0} = - \frac{2}{\pi} \,\mathrm{Re}\,\int_{\m{H}} \nu(z) \bigg(\frac{g_t'(z)^2}{g_t(z)-\lambda_t} - \frac{1}{z}\bigg) \dd^2 z 
    \end{equation}
    and 
    \begin{equation}
        \frac{\partial a_t^\enu}{\partial \vare}\bigg|_{\vare=0} = \frac{1}{\pi} \,\mathrm{Re}\int_{\m{H}} \nu(z) \big(g_t'(z)^2 - 1\big) \dd^2 z 
    \end{equation}
    where $\dd^2 z$ is the Euclidean area measure, $\lambda_\cdot$ is the driving function of $\eta$, $g_\cdot$ is the Loewner chain of $\eta$.
\end{thm}

Our proof relies on the simple but crucial observation that the Loewner driving function and the capacity parametrization of the curve can be expressed by the pre-Schwarzian and Schwarzian derivatives, respectively, of well-chosen maps (Lemma~\ref{lem:preS_driving}). 
   
   \bigskip 

   We extend our considerations to the Loewner driving function associated with a Jordan curve $\gamma \subset \Chat = \m C \cup \{\infty\}$, now defined on $\m R$ instead of $\m R_+$.  The loop driving function was defined in \cite{W2} and can be thought of as a consistent family of chordal Loewner driving functions. 
   See Section~\ref{sec:driving} for the precise definition. 
  We point out that for a given Jordan curve, there are a few choices we make to define its driving function $t\mapsto \l_t$:
  \begin{itemize}
      \item  the orientation of $\gamma$; 
      \item  a point on $\gamma$  called the \emph{root}, which we denote by $\g (-\infty) = \g (+\infty)$ (we also use $\g (\pm\infty)$ when we do not emphasize the difference between the start point and the end point of the parametrization);
      \item  another point on $\g$, which we call $\g (0)$;
      \item a conformal map $H_0 : \Chat \smallsetminus \g [-\infty, 0] \to \m C \smallsetminus \m R_+$, such that $H_0 (\g (0)) = 0$ and $H_0 (\g (+\infty)) = \infty$, where  $\g [-\infty, 0]$ denotes the closed subinterval of $\g$ (as a set) going from the root to $\g (0)$ following the orientation of the curve.
  \end{itemize}
Then, we can complete the continuous parametrization of $\g$ on $(-\infty,0) \cup (0,+\infty)$ in a unique way
such that for each $s \in \m{R}$, the chord $\gamma(\cdot+s)$ traverses the simply connected domain $\Chat \smallsetminus \gamma[-\infty,s]$ in capacity parametrization.\footnote{For the simplicity of notation, we assumed here that the capacity parametrization of $\g$ is bi-infinite. See Footnote~\ref{footnote:capacity} for further details regarding this assumption.} Moreover, the chordal driving function of  $\gamma(\cdot+s)$  in $\Chat \smallsetminus \gamma[-\infty,s]$ is given by $\l_{\cdot + s} - \l_{s}$ (see Lemma~\ref{lem:loop-time-change}). 

If the orientation and the root of $\g$ are fixed, different choices of $\g (0)$ and $H_0$ result in changes to the driving function of the form 
    \begin{equation} \label{eq:driving-function-scaling}
        \tilde \lambda_t =  c \left(\lambda_{c^{-2} (t + s)} - \lambda_{c^{-2} s}\right)
    \end{equation}
  for some $c > 0$ and $s \in \m R$. Such transformations do not change the Dirichlet energy of $\lambda$. 
    Rather surprisingly, the Dirichlet energy of the loop driving function does not depend on the choice of the root or the orientation either, as shown in \cite{W1,RW}. These symmetries are further explained by the following theorem.
  \begin{thm}[See \cite{W2}] \label{thm:intro_equiv}
      The Loewner energy of $\g$, defined as 
   \begin{equation} \label{eq:Loewner_energy_df}
       I^L(\gamma) = \frac{1}{2} \int_{-\infty}^{+\infty} |\dot \lambda_t|^2 \,\dd t 
   \end{equation} 
    equals $1/\pi$ times the \emph{universal Liouville action} ${\bf S}$ introduced by Takhtajan and Teo in \textnormal{\cite{TT06}}, defined as
\begin{equation}\label{eq_def_S1}
{\bf S} (\g) : = \int_{\m D} \abs{\frac{f''}{f'}(z)}^2 \dd^2 z + \int_{\m D^*}\abs{\frac{g''}{g'}(z)}^2 \dd^2 z + 4\pi \log \abs{\frac{f'(0)}{g'(\infty)}}.
\end{equation}
Here, $f : \m D \to \O$ and $g : \m D^* \to \O^*$ are conformal maps such that $g (\infty) = \infty$,  $\O$ and $\O^*$ are  respectively the bounded and unbounded connected components of $\m C \smallsetminus \g$, and $g'(\infty) = \lim_{z\to \infty} g'(z)$. If $\g$ passes through $\infty$, we replace $\g$ by $A(\g)$ where $A$ is any M\"obius transformation of $\Chat$ sending $\g$ to a bounded curve.
  \end{thm}

\begin{remark}
    Although it may not be so apparent from \eqref{eq_def_S1}, it will follow immediately from the definition of the loop driving function and \eqref{eq:Loewner_energy_df} that $I^L$ is invariant under M\"obius transformations of $\Chat$. See Remark~\ref{rem:mobius_inv_dr}.
    A Jordan curve for which ${\bf S}$ is finite is called a \emph{Weil--Petersson quasicircle}.
\end{remark}

Using Theorem~\ref{thm:intro_driving_function}, we obtain in Section~\ref{sec:WP} the following first variation formula of the Loewner energy. This formula coincides with that of the universal Liouville action ${\bf S}$ in \cite[Ch.\,2, Thm.\,3.8]{TT06} divided by $\pi$, thus giving another explanation of the identity $I^L = {\bf S}/\pi$.  This variational formula was crucial in \cite{TT06} to show that $\bf S$ is a K\"ahler potential of the Weil--Petersson Teichm\"uller space.

\begin{thm}\label{thm:intro_energy-derivative-schwarzian}
Let $\mu\in L^\infty(\m{C})$ be an infinitesimal Beltrami differential with compact support in $\Chat \smallsetminus\gamma$. For $\vare \in \m{R}$ with $\|\vare \mu\|_{\infty} < 1$, let $\omega^\emu: \Chat \to \Chat$ be any quasiconformal mapping with Beltrami coefficient $\emu$.
Let $\g^{\vare \mu} = \omega^{\vare \mu} (\g)$. Then,
\begin{equation} 
    \frac{\dd}{\dd \vare} \bigg|_{\vare=0} I^L(\gamma^\emu)= -\frac{4}{\pi} \, \mathrm{Re} \left[ \int_{\O} \mu(z) \mc S [f^{-1}](z) \,\dd^2 z+ \int_{\O^*} \mu(z) \mc S [g^{-1}](z) \,\dd^2 z \right],
\end{equation}
where $f$ and $g$ are the conformal maps  in Theorem~\ref{thm:intro_equiv} and $\mc S[\varphi] = \varphi'''/\varphi' - (3/2)(\varphi''/\varphi')^2$ is the Schwarzian derivative of $\varphi$.
\end{thm}

In the language of conformal field theory, this theorem states that the holomorphic stress-energy tensor of the Loewner energy is given by a multiple of the Schwarzian derivative of the uniformizing map on each complementary component of the curve. 

\begin{remark} 
The Loewner energy of $\g^{\vare \mu}$ does not depend on the choice of the solution $\omega^{\vare \mu}$ to the Beltrami equation, as all such solutions are equivalent up to post-compositions by M\"obius transformations of $\Chat$.
        In \cite{TT06}, $\mu$ is an $L^2$-harmonic Beltrami differential supported on only one side of the curve $\g$. 
       Here, we allow the support of $\mu$ to be on both sides of $\g$ but require it to be disjoint from $\g$.
\end{remark}

Since the support of $\mu$ is away from $\g$, there exists a (not necessarily simply connected) domain $\domain$ containing $\g$ such that $\domain \cap \text{supp}(\mu) = \varnothing$. In particular, $\o^{\emu}$ is conformal in $\domain$. In \cite[Thm.\,4.1]{W3}, the second author showed that the change of the Loewner energy under a conformal map in the neighborhood of the curve could be expressed in terms of the $\SLE_{8/3}$ loop measure introduced in \cite{werner_measure}, which is the induced measure obtained by taking the outer boundary of a loop under Brownian loop measure \cite{LSW_CR_chordal,LW2004loupsoup}.
Combining this with Theorem~\ref{thm:intro_energy-derivative-schwarzian}, we immediately obtain the following variational formula for the $\SLE_{8/3}$ loop measure. 
\begin{cor}
For every domain $\domain$ containing $\g$ such that $\domain \cap \textnormal{supp}(\mu) = \varnothing$,
we have
    \begin{equation} 
    \frac{\dd}{\dd \vare}\bigg|_{\vare=0} \mc W (\gamma^\emu, \o^{\emu} (\domain)^c;\Chat) = \frac{1}{3\pi} \, \mathrm{Re} \left[ \int_{\O} \mu(z) \mc S [f^{-1}](z) \,\dd^2 z+ \int_{\O^*} \mu(z) \mc S [g^{-1}](z) \,\dd^2 z \right]
\end{equation}
where $\mc W (\gamma^\emu, \o^{\emu} (\domain)^c;\Chat)$ denotes the total mass of loops intersecting both $\gamma^\emu$ and the complement of $\o^{\emu} (\domain)$ under the $\SLE_{8/3}$ loop measure on $\Chat$. 
\end{cor}

\section{Deformation of chords in the half-plane} \label{sec:chord}

\subsection{Variation of the chordal Loewner driving function} \label{sec:chordal-variation}
Let $\eta:(0,+\infty) \to \m{H}$ be a continuously parametrized simple chord from $0$ to $\infty$. For general $\eta$, there exists a unique conformal map $\m{H}\smallsetminus \eta(0,t] \to \m{H}$ with the expansion 
\begin{equation*}
    z + \frac{2a_t}{z} + O\bigg(\frac{1}{z^2}\bigg) \quad \text{as} \;\; z\to \infty
\end{equation*}
where $a_t>0$ is a constant known as (one-half of) the \textit{half-plane capacity} of the curve $\eta[0,t]$. We call $t\mapsto a_t$ the \emph{capacity function} of $\eta$. Let us denote $T_+:= \lim_{t\to+\infty} a_t \in (0,+\infty]$.
\footnote{See \cite[Thm.\,1]{LLN_capacity} for an example where $T_+ < \infty$.}  After an appropriate time change, we may assume that $\eta$ is parametrized by its (one-half) half-plane capacity: i.e., $a_t=t$. The family of uniformizing maps $g_t: \m{H}\smallsetminus \eta(0,t] \to \m{H}$ satisfying
\begin{equation*}
    g_t(z) = z + \frac{2t}{z} + O\bigg(\frac{1}{z^2}\bigg) \quad \text{as} \;\; z\to \infty
\end{equation*}
under this \textit{capacity parametrization} is called the \emph{Loewner chain} corresponding to $\eta$.
The function $[0,T_+)\to\m{R}$, $t\mapsto \lambda_t := g_t(\eta(t))$ is called the \emph{driving function} of the curve $\eta$. 

We consider deformations of $\eta$ under quasiconformal self-maps of $\m{H}$. Let $\nu \in L^\infty(\m{H})$ be a complex-valued function with  compact support in $\m{H}\smallsetminus \eta$. The measurable Riemann mapping theorem states that for $\vare\in\m{R}$ with $ |\vare| < 1/\|\nu\|_\infty$, there exists a unique quasiconformal self-map $\psi^\enu:\m{H}\to \m{H}$ which solves the Beltrami equation
\begin{equation}
    \partial_{\bar z} \psi^\enu = (\vare\nu) \partial_z \psi^\enu
\end{equation}
and has $\psi^\enu(0) = 0$ and $\psi^\enu(z) - z = O(1)$ as $z\to\infty$. We adopt the following notations, illustrated in Figure~\ref{fig:qc-maps}.
\begin{itemize}
    \item Denote the deformed chord by $\eta^\enu:= \psi^\enu \circ \eta$.
    \item Let $g_t^\enu: \m{H}\smallsetminus \eta^\enu(0,t] \to \m{H}$ be the Loewner chain associated with the deformed curve $\eta^\enu[0,t]$. 
    \item Denote the driving function of $\eta^\enu$ by $\lambda_t^\enu := g_t^\enu(\eta^\enu(t))$.
    \item Note that $\eta^\enu$ is not necessarily parametrized by its half-plane capacity. Denote the capacity function of $\eta^\enu[0,t]$ by $a_t^\enu$, so that $g_t^\enu(z) = z + 2a_t^\enu z^{-1} + O(z^{-2})$ as $z\to\infty$.
\end{itemize}

 \begin{figure}
    \centering
    \includegraphics[width=0.95\textwidth]{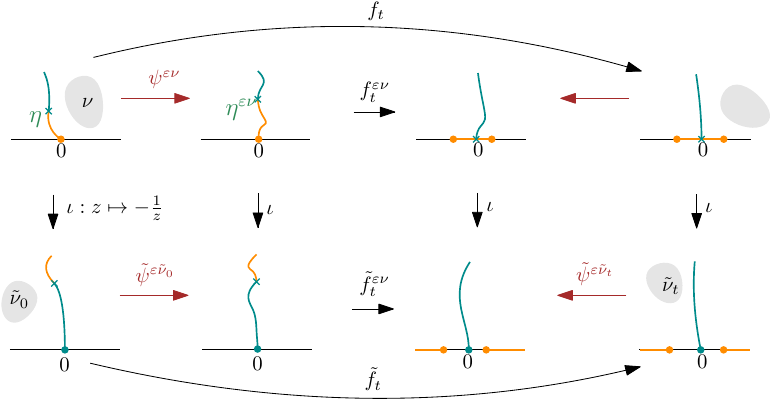}
    \caption{A commutative diagram illustrating the quasiconformal maps and related Loewner chains in Section~\ref{sec:chordal-variation}. The gray shaded areas denote the support of the Beltrami differentials. The arrows in red are quasiconformal maps, and those in black are conformal maps.}
    \label{fig:qc-maps}
\end{figure}

This section aims to prove Theorem~\ref{thm:intro_driving_function}. For this, we first express $\lambda_t^\enu$ and $a_t^\enu$ in terms of the pre-Schwarzian and Schwarzian derivatives of an appropriately conjugated Loewner chain (Lemma~\ref{lem:preS_driving}). We then find the first variations of these derivatives using the measurable Riemann mapping theorem (Proposition~\ref{prop:driving-fn-derivation}).

The \textit{centered Loewner chain} 
\begin{equation}
    f_t(z) : = g_t(z) - \lambda_t, 
\end{equation} 
satisfies $f_t(\eta(t))=0$ and 
\begin{equation}\label{eq:f_eq}
    f_t(z) = z - \lambda_t + \frac{2t}{z} + O\bigg(\frac{1}{z^2}\bigg) \qquad \text{as } z\to \infty.
\end{equation}
Let $\inversion (z) := -1/z$ be the inversion map. Define the \textit{inverted Loewner chain} by
\begin{equation}\label{eq:def-tilde f_t}
    \tilde f_t(z):= \inversion \circ f_t \circ \inversion (z) = -\frac{1}{f_t(-1/z)}. 
\end{equation}
Then, $\tilde f_t: \m{H} \smallsetminus (\inversion \circ \eta (0,t]) \to \m{H}$ is the uniformizing map with normalization $\tilde f_t(0) = 0$, $\tilde f_t'(0) = 1$, and $\tilde f_t ( \inversion \circ \eta(t)) = \infty$. Combining the expansion \eqref{eq:f_eq} of $f_t$ at $\infty$ with \eqref{eq:def-tilde f_t}, we see that as $z\to 0$,
\begin{equation}\label{eq:tilde f_t-Taylor-series}
    \tilde f_t(z) = -\frac{1}{-z^{-1} -\lambda_t - 2tz + O(z^2)} = z  - \lambda_t z^2 + (\lambda_t^2 - 2t) z^3 + O(z^4).
\end{equation}

Similarly, define 
\begin{equation}
    f_t^\enu(z):= g_t^\enu(z) - \lambda_t^\enu \quad \text{and} \quad \tilde f_t^\enu(z):= \inversion \circ f_t^\enu \circ \inversion(z).
\end{equation}
A calculation analogous to \eqref{eq:tilde f_t-Taylor-series} using the series expansion of $g_t^\enu$ at $\infty$ leads to
\begin{equation}\label{eq:tilde f_t^enu-Taylor-series}
    \tilde f_t^\enu(z) = z - \lambda_t^\enu z^2 + ((\lambda_t^\enu)^2 - 2a_t^\enu)z^3 + O(z^4) \quad \text{as} \;\; z\to 0.
\end{equation}
By the Schwarz reflection principle, $\tilde f_t$ and $\tilde f_t^\enu$ extend respectively to conformal maps on $\m{C} \smallsetminus \inversion(\eta(0,t] \cup \overline{\eta(0,t]})$ and $\m{C} \smallsetminus \inversion(\eta^\enu(0,t] \cup \overline{\eta^\enu(0,t]})$, where $\overline{\cdot}$ denotes the complex conjugate. In particular, they are conformal in some neighborhood of 0. 

Recall that the \textit{pre-Schwarzian} (also known as non-linearity) and \textit{Schwarzian derivatives} of a conformal map $\varphi$ are, respectively, 
\begin{equation}
    \Nonlinearity \varphi = \frac{\varphi''}{\varphi'} \quad \text{and} \quad \Schwarzian \varphi = \frac{\varphi'''}{\varphi'} - \frac{3}{2}\bigg(\frac{\varphi''}{\varphi'}\bigg)^2.
\end{equation}
The chain rules for the pre-Schwarzian and Schwarzian derivatives are
\begin{equation}\label{eq:chain_rule}
    \Nonlinearity [f\circ g] = ((\Nonlinearity f) \circ g)g' + \Nonlinearity g \quad \text{ and } \quad \Schwarzian [f\circ g] = ((\Schwarzian f)\circ g) (g')^2 + \Schwarzian g
\end{equation}
for any conformal maps $f$ and $g$ such that $f\circ g$ is well-defined.

\begin{lem}\label{lem:preS_driving}
Consider $\tilde f_t$ and $\tilde f_t^\enu$ as conformal maps extended by reflection to a neighborhood of 0. Then,
\begin{equation}\label{eq:derivatives-drivingfn-capacity}
    \lambda_t = -\frac{1}{2} \Nonlinearity \tilde f_t(0),\quad \lambda_t^\enu = -\frac{1}{2} \Nonlinearity \tilde f_t^\enu(0), \quad \text{and} \quad a_t^\enu = -\frac{1}{12} \Schwarzian \tilde f_t^\enu(0).
\end{equation}
\end{lem}
\begin{proof}
    The lemma follows from inspecting the coefficients of \eqref{eq:tilde f_t-Taylor-series} and \eqref{eq:tilde f_t^enu-Taylor-series}.
\end{proof}

Let $\vare \tilde \nu_t$ be the Beltrami coefficient of the quasiconformal map
\begin{equation}
    \tilde \psi^{\vare \tilde \nu_t} := \tilde f_t^\enu \circ \inversion \circ \psi^\enu \circ \inversion \circ \tilde f_t^{-1} = \inversion \circ f_t^\enu \circ \psi^\enu \circ f_t^{-1} \circ \inversion.
\end{equation}
In particular, $\tilde \psi^{\vare \tilde \nu_0} = \inversion\circ \psi^\enu \circ \inversion$ is the quasiconformal map which deforms $\inversion \circ \eta$ to $\inversion \circ \eta^{\enu}$. Note that $\tilde \psi^{\vare \tilde \nu_t}$ is conformal in a neighborhood of 0 and satisfies $\tilde \psi^{\vare \tilde \nu_t}(0) = 0$, $(\tilde \psi^{\vare \tilde \nu_t})'(0) = 1$, and $\tilde \psi^{\vare \tilde \nu_t}(\infty) = \infty$. 

\begin{prop}\label{prop:driving-fn-derivation}
Let $\tilde\nu_t$ be the Beltrami differential defined above. 
Then, 
\begin{equation}\label{eq:driving-function-perturbed}
    \begin{aligned}
    \frac{\partial \lambda_t^\enu}{\partial \vare}\bigg|_{\vare=0} &=  \frac{2}{\pi} \, \mathrm{Re} \int_{\m{H}} \frac{\tilde\nu_t(z)-\tilde\nu_0(z)}{z^3} \,\dd^2 z  \\&= \frac{2}{\pi} \, \mathrm{Re} \int_{\m{H}} \tilde\nu_0(z) \bigg(\frac{\tilde f_t'(z)^2}{\tilde f_t(z)^3} - \frac{1}{z^3} \bigg) \,\dd^2 z \\
    &=  -\frac{2}{\pi} \, \mathrm{Re} \int_{\m{H}} \nu(z) \bigg(\frac{f_t'(z)^2}{f_t(z)} - \frac{1}{z}\bigg) \,\dd^2 z 
    \end{aligned}
\end{equation}    
and
\begin{equation}\label{eq:capacity-perturubed}
    \begin{aligned}
    \frac{\partial a_t^\enu}{\partial t}\bigg|_{\vare=0} &= \frac{1}{\pi} \, \mathrm{Re} \int_{\m{H}} \frac{\tilde\nu_t(z)-\tilde\nu_0(z)}{z^4} \,\dd^2 z \\ 
    &=  \frac{1}{\pi} \, \mathrm{Re} \int_{\m{H}} \tilde\nu_0(z) \bigg(\frac{\tilde f_t'(z)^2}{\tilde f_t(z)^4} - \frac{1}{z^4} \bigg) \,\dd^2 z \\
    &= \frac{1}{\pi} \, \mathrm{Re} \int_{\m{H}} \nu(z) \big(f_t'(z)^2 - 1 \big) \,\dd^2 z .
    \end{aligned}
\end{equation}
\end{prop}

\begin{proof}[Proof of Theorem~\ref{thm:intro_driving_function}]
    It suffices to substitute $f_t(z) = g_t(z) -\lambda_t$ in Proposition~\ref{prop:driving-fn-derivation}.
\end{proof}
\begin{proof}[Proof of Proposition~\ref{prop:driving-fn-derivation}]
We can extend $\tilde \psi^{\vare \tilde \nu_t}$ to a quasiconformal self-map of the Riemann sphere $\Chat = \m{C} \cup \{\infty\}$ by reflecting it with respect to the real axis. The Beltrami coefficient for this extension of $\tilde \psi^{\vare \tilde \nu_t}$ is $\vare \hat \nu_t$ where
\begin{equation}\label{eq:nu_t-reflection}
    \hat \nu_t (z):= \begin{cases}
        \tilde\nu_t(z) & \text{if } z \in \m{H}, \\
        0 & \text{if } z\in \m{R}, \\
        \overline{\tilde\nu_t(\overline{z})} & \text{if } z\in \m{H}^*.
    \end{cases}
\end{equation}
Then, by the measurable Riemann mapping theorem, 
\begin{equation}\label{eq:beltrami-solution-halfplane}
    \tilde \psi^{\vare \tilde \nu_t}(\zeta) = \zeta - \frac{\vare}{\pi}\int_{\m{C}} \hat \nu_t(z) \left(\frac{1}{z-\zeta} - \frac{1}{z} - \frac{\zeta}{z^2} \right) \dd^2 z + o(\vare)
\end{equation}
locally uniformly in $\zeta\in\m{C}$ as $\vare \to 0$. Moreover, since $\partial_\vare$ commutes with $\partial_\zeta$ when applied to $\tilde \psi^{\vare \tilde \nu_t}$ and $\hat \nu_t$ has a compact support in $\m{C}\smallsetminus\{0\}$, we have
\begin{align}\label{eq:nu_t-nonlinearity}
    \frac{\partial}{\partial \vare} \bigg|_{\vare=0} \Nonlinearity \tilde \psi^{\vare \tilde \nu_t}(0) &= -\frac{2}{\pi} \int_{\m{C}} \frac{\hat \nu_t(z)}{z^3} \,\dd^2 z = -\frac{4}{\pi} \, \mathrm{Re}\int_{\m{H}} \frac{\tilde\nu_t(z)}{z^3} \dd^2 z, \\ \label{eq:nu_t-schwarzian}
    \frac{\partial}{\partial \vare} \bigg|_{\vare=0} \Schwarzian \tilde \psi^{\tilde \nu_t}(0)&= -\frac{6}{\pi} \int_{\m{C}} \frac{\hat \nu_t (z)}{z^4} \,\dd^2 z = -\frac{12}{\pi} \, \mathrm{Re} \int_{\m{H}} \frac{\tilde\nu_t(z)}{z^4} \,\dd^2 z.   
\end{align}
Since $\tilde f_t^\enu = \tilde \psi^{\vare \tilde \nu_t} \circ \tilde f_t \circ (\tilde \psi^{\vare \tilde \nu_0})^{-1}$ and $\tilde \psi^{\vare \tilde \nu_t}(z)$, $\tilde f_t(z)$, and $\tilde \psi^{\vare \tilde \nu_0}(z)$ all behave as $z + o(z)$ as $z\to 0$, we have from Lemma~\ref{lem:preS_driving} and the chain rules \eqref{eq:chain_rule} that
\begin{align} 
\label{eq:drivingfn-chain-rule}\begin{split}
    -2\lambda_t^\emu &= \Nonlinearity \tilde f_t^\emu(0) = \Nonlinearity \tilde \psi^{\vare \tilde \nu_t}(0) + \Nonlinearity \tilde f_t(0) - \Nonlinearity \tilde \psi^{\vare \tilde \nu_0}(0) \\
    &= \Nonlinearity \tilde \psi^{\vare \tilde \nu_t}(0) - 2\lambda_t - \Nonlinearity \tilde \psi^{\vare \tilde \nu_0}(0),
\end{split}\\
\label{eq:capacity-chain-rule}\begin{split}
    -12a_t^\emu & = \Schwarzian \tilde f_t^\emu(0) = \Schwarzian \psi_t^{\vare\nu_t}(0) + \Schwarzian \tilde f_t(0) - \Schwarzian \psi^{\vare\nu_0}(0) \\
    &= \Schwarzian \psi_t^{\vare\nu_t}(0) -12 t - \Schwarzian \psi^{\vare\nu_0}(0) .
\end{split}
\end{align}
Combining these with \eqref{eq:nu_t-nonlinearity} and \eqref{eq:nu_t-schwarzian}, we obtain the first equalities in \eqref{eq:driving-function-perturbed} and \eqref{eq:capacity-perturubed}.

Observe that $\tilde \psi^{\vare \tilde \nu_t} = \tilde f_t^\enu \circ \tilde \psi^{\vare \tilde \nu_0} \circ \tilde f_t^{-1}$, where $\tilde f_t^\enu$ and $\tilde f_t^{-1}$ are conformal maps. Hence, by the composition rule for Beltrami coefficients,
\begin{equation}\label{eq:nu_t-coord-change}
    \tilde\nu_t(\tilde f_t(z)) = \tilde\nu_0(z) \frac{\tilde f_t'(z)^2}{|\tilde f_t'(z)|^2}.
\end{equation} 
Substituting \eqref{eq:nu_t-coord-change} into \eqref{eq:nu_t-nonlinearity} and \eqref{eq:nu_t-schwarzian}, we have that 
\begin{align}
    \label{eq:nu_0-nonlinearity} \frac{\partial}{\partial \vare} \bigg|_{\vare=0} \Nonlinearity \tilde \psi^{\vare \tilde \nu_t}(0) &= -\frac{4}{\pi}\,\mathrm{Re}\int_{\m{H}} \tilde\nu_0(z)\frac{\tilde f_t'(z)^2}{\tilde f_t(z)^3}\,\dd^2 z ,\\ 
    \label{eq:nu_0-schwarzian} \frac{\partial}{\partial \vare} \bigg|_{\vare=0} \Schwarzian \tilde \psi^{\vare \tilde \nu_t}(0) &= -\frac{12}{\pi}\,\mathrm{Re}\int_{\m{H}} \tilde\nu_0(z)\frac{\tilde f_t'(z)^2}{\tilde f_t(z)^4}\,\dd^2 z .
\end{align}
We thus have the second equalities in \eqref{eq:driving-function-perturbed} and \eqref{eq:capacity-perturubed}.

Finally, recall that $\tilde \psi^{\vare \tilde \nu_0} = \inversion \circ \psi^\enu \circ \inversion$. 
Since the inversion map $\inversion(\zeta)=-1/\zeta$ is conformal, 
\begin{equation}
    \tilde\nu_0(-1/\zeta) = \nu(\zeta) \frac{|\zeta|^4}{\zeta^4}.
\end{equation}

Recall that $\tilde f_t(z) = -1/f_t(-1/z)$, and hence $\tilde f_t'(z) = f_t'(-1/z)/(zf_t(-1/z))^2$.
Substituting $z = -1/\zeta$ in \eqref{eq:nu_0-nonlinearity} and \eqref{eq:nu_0-schwarzian}, we obtain the final equalities in \eqref{eq:driving-function-perturbed} and \eqref{eq:capacity-perturubed}.
\end{proof}

\subsection{Variation of chordal Loewner energy}

Let $\eta : (0,T_+) \to \m H$ be a simple chord from $0$ to $\infty$ parametrized by half-plane capacity (\ie $a_t = t$) 
and $t\mapsto \l_t$ be its Loewner driving function. The \emph{Loewner energy} of $\eta$ (resp. the partial Loewner energy of $\eta$ up to time $T\in(0,T_+)$) is 
$$I^C(\eta) = \frac12\int_0^{T_+} \dot \lambda_t ^2 \, \dd t \qquad \text{resp.} \qquad I^C(\eta(0,T]) = \frac12\int_0^T \dot \lambda_t ^2 \, \dd t$$
if $\lambda$ is absolutely continuous and $\dot \lambda$ is its almost everywhere defined derivative with respect to $t$. We set $I^C(\eta) = \infty$ if  $\lambda$ is not absolutely continuous. If $I^C(\eta) < +\infty$, then $T_+ = +\infty$; see \cite[Thm.\ 2.4]{sle_ld_survey}.

We also define the Loewner energy of a simple chord $\eta$ in a simply connected domain $D$ connecting two distinct prime ends $a,b$ as 
$$I^C_{D;a,b} (\eta) : = I^C (\varphi(\eta))$$
where $\varphi$ is any conformal map $D \to \m H$ with $\varphi(a) =0$ and $\varphi(b) =\infty$. The partial Loewner energy in $(D;a,b)$ is defined similarly.

When $\lambda_t$ is absolutely continuous, we can compute the first variations of $\dot \lambda_t^\enu$ and $\dot a_t^\enu$.

\begin{prop}\label{prop:d/dt lambda-variation}
    For all $\vare \in (-1/\|\nu\|_\infty,1/\|\nu\|_\infty)$, the functions 
    $t\mapsto \lambda_t^\enu - \lambda_t$ and $t\mapsto a_t^\enu$ 
    are continuously differentiable. Furthermore, if $t\mapsto\lambda_t$ is absolutely continuous, then $t\mapsto \lambda_t^\enu$ is also absolutely continuous and, for almost every $t$,
    \begin{equation}
       \frac{\partial \dot \lambda_t^\enu}{\partial \vare} \bigg|_{\vare=0} = \frac{1}{\pi}\,\mathrm{Re} \int_{\m{H}} \nu(z) \bigg(12\frac{f_t'(z)^2}{f_t(z)^3}-2\dot \lambda_t\frac{f_t'(z)^2}{f_t(z)^2} \bigg) \,\dd^2 z
    \end{equation}
    and
    \begin{equation}
       \frac{\partial \dot a_t^\enu}{\partial \vare} \bigg|_{\vare=0} = -\frac{4}{\pi}\,\mathrm{Re} \int_{\m{H}} \nu(z) \frac{f_t'(z)^2}{f_t(z)^2} \, \dd^2 z.
    \end{equation}
\end{prop}
\begin{proof}
    From the Loewner equation $\partial_t g_t(z) = 2/(g_t(z)-\lambda_t)$, we have 
    \begin{equation*}
        \partial_t (f_t(z)+\lambda_t) = \frac{2}{f_t(z)} \quad \text{and} \quad
        \partial_t f_t'(z) = -\frac{2f_t'(z)}{f_t(z)^2}.
    \end{equation*}
    Recalling $\tilde f_t(z) = -1/f_t(-1/z)$, it follows that $\tilde f_t'(z)$ is continuously differentiable in $t$.
    From \eqref{eq:nu_t-reflection} and \eqref{eq:nu_t-coord-change}, we see that $(\vare,t) \mapsto \vare \hat \nu_t$ is continuously differentiable. Then,
    $\lambda_t^\enu - \lambda_t = -\frac{1}{2}(\Nonlinearity \tilde \psi^{\vare \tilde \nu_t}(0) - \Nonlinearity \tilde \psi^{\vare \tilde \nu_0}(0))$ and $a_t^\enu - t = -\frac{1}{12}(\Schwarzian \tilde \psi^{\vare \tilde \nu_t}(0) - \Schwarzian \tilde \psi^{\vare \tilde \nu_0}(0))$ are continuously differentiable in the same variables \cite{AhlforsBers}.
    
    We can check directly from the integral representations of $\partial_\vare \Nonlinearity\tilde \psi^{\vare\tilde\nu_t}(0)$ and $\partial_\vare \Schwarzian\tilde\psi^{\vare\tilde\nu_t}(0)$ that they are continuously differentiable in $t$. If $\lambda_t$ is absolutely continuous, then we have from \eqref{eq:driving-function-perturbed} that for almost every $t$,
    \begin{equation*}\begin{aligned}
        \frac{\partial \dot \lambda_t^\enu}{\partial \vare} \bigg|_{\vare=0} &= -\frac{1}{2}\frac{\partial^2 (\Nonlinearity \tilde \psi^{\vare\tilde\nu_t}(0))}{\partial \vare \partial t} \bigg|_{\vare=0} = -\frac{1}{2} \frac{\partial^2 (\Nonlinearity \tilde \psi^{\vare\tilde \nu_t}(0))}{\partial t \partial \vare}\bigg|_{\vare=0} \\& = -\frac{2}{\pi} \,\mathrm{Re}  \int_{\m{H}} \nu(z) \partial_t \bigg(\frac{f_t'(z)^2}{f_t(z)} \bigg) \,\dd^2 z .
        \end{aligned}
    \end{equation*}
    Similarly, \eqref{eq:capacity-perturubed} implies
    \begin{equation*}
        \frac{\partial \dot a_t^\enu}{\partial \vare} \bigg|_{\vare=0} = \frac{1}{\pi} \,\mathrm{Re} \int_{\m{H}} \nu(z) \,\partial_t \big(f_t'(z)^2\big) \,\dd^2 z.
    \end{equation*}

    Using the formulas for $\partial_t f_t$ and $\partial_t f_t'$ above, we have 
    \begin{equation*}
        \partial_t \big(f_t'(z)^2\big) = -4\frac{f_t'(z)^2}{f_t(z)^2} \quad \text{and} \quad \partial_t \,\bigg(\frac{f_t'(z)^2}{f_t(z)}\bigg) = -6\frac{f_t'(z)^2}{f_t(z)^3} + \dot \lambda_t \frac{f_t'(z)^2}{f_t(z)^2}.
    \end{equation*}
   This completes the proof.
\end{proof}

Proposition~\ref{prop:d/dt lambda-variation} allows us to compute the first variation of the chordal Loewner energy for a finite portion of the chord $\eta^\enu$.

\begin{cor} \label{cor:chord-LoewnerE}
    Let $T\in (0,T_+)$. Suppose $\lambda_t$ is absolutely continuous on $[0,T]$ and $\dot \lambda_t \in L^2([0,T])$. Then,
    \begin{equation}\label{eq:chordal-energy-variation}
            \frac{\partial}{\partial \vare} \bigg|_{\vare=0} I^C(\eta^\enu(0,T])  = \frac{12}{\pi} \,\mathrm{Re} \int_{\m{H}} \nu(z) \int_0^T  \dot \lambda_t  \frac{f_t'(z)^2}{f_t(z)^3}\,\dd t\,\dd^2 z.
    \end{equation}
\end{cor}

\begin{proof}
    From Proposition~\ref{prop:d/dt lambda-variation}, we see that $(\dot \lambda_t^\enu)^2/\dot a_t^\enu$ is integrable on $[0,T]$ whenever $|\vare|<1/\|\nu\|_\infty$. 
    Moreover, we obtain that for almost every $t\in[0,T]$,
    \begin{equation*}
        \frac{\partial}{\partial \vare} \bigg|_{\vare=0} \frac{(\dot \lambda_t^\enu)^2}{\dot a_t^\enu}  = 2\dot \lambda_t \frac{\partial \dot \lambda_t^\enu}{\partial \vare} \bigg|_{\vare=0} - \dot \lambda_t^2 \frac{\partial \dot a_t^\enu}{\partial \vare} \bigg|_{\vare=0} = \frac{24\dot \lambda_t }{\pi} \,\mathrm{Re} \int_{\m{H}} \nu(z) \frac{f_t'(z)^2}{f_t(z)^3}  \,\dd^2 z.
    \end{equation*}
    Since $\nu$ is compactly supported in $\m{H}\smallsetminus\eta$, the integral on the right-hand side is continuous in $t$ and hence bounded on $[0,T]$. Using the Leibniz integral rule, we conclude that
    \begin{equation*}\begin{aligned}
        I^C(\eta^\enu(0,T]) &= \frac{1}{2}\int_0^T \left|\frac{\dd \lambda_t^\enu}{\dd a_t^\enu}\right|^2 \dd a_t^\enu = \frac{1}{2} \int_0^T \frac{(\dot \lambda_t^\enu)^2}{\dot a_t^\enu}\,\dd t \\ 
        & = I^C(\eta[0,T]) + \vare \int_0^T \frac{12\dot \lambda_t }{\pi} \,\mathrm{Re} \left[\int_{\m{H}} \nu(z) \frac{f_t'(z)^2}{f_t(z)^3}  \dd^2 z \right] \dd t + o(\vare)\end{aligned}
    \end{equation*}
    as $\vare\to 0$.
\end{proof}

 \section{Deformation of Weil--Petersson quasicircles}\label{sec:WP}

 In this section,   we consider the Loewner chain for a Jordan curve by conjugating the chordal Loewner chain by $z \mapsto z^2$. 
 This simple 
 operation  relates the integrand in  \eqref{eq:chordal-energy-variation} to a Schwarzian derivative (Proposition~\ref{prop:infinitesimal-loewner-energy}) which leads to the proof of Theorem~\ref{thm:intro_energy-derivative-schwarzian}.

\textit{Convention.} 
We take the branch of the complex square root function $\sqrt{z}$ (or $z^{1/2}$) on $\m C$ to be the one with the image in $\m{H} \cup \m{R}_+$.

\subsection{Loop driving function} \label{sec:driving}

We first recall the definition of the Loewner driving function for a Jordan curve. 
See Figure~\ref{fig:qc-loop} for an illustration of the maps used in Section~\ref{sec:WP}.

 \begin{figure}[ht]
 \centering\includegraphics[width=0.95\textwidth]{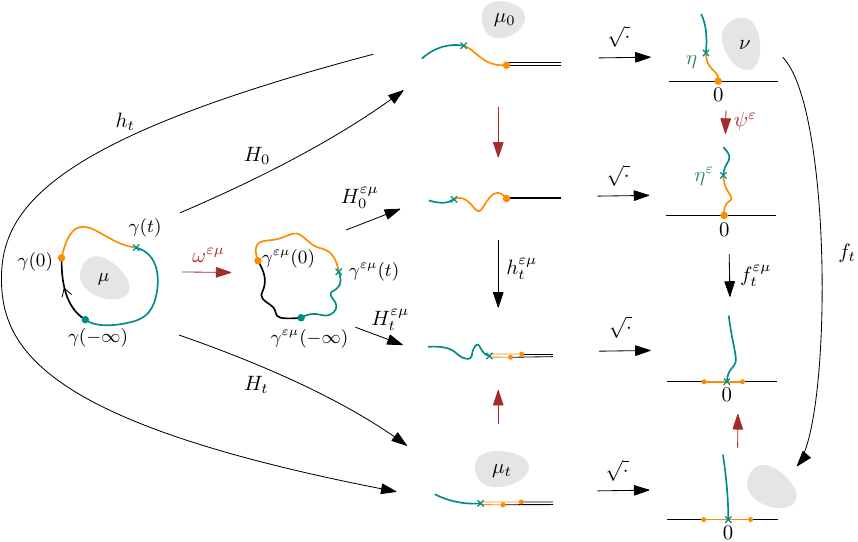}
    \caption{A commutative diagram illustrating the quasiconformal maps and related conformal mapping-out functions in Section~\ref{sec:WP}. The gray shaded areas denote the support of the Beltrami differentials. The arrows in red are quasiconformal maps, and those in black are conformal maps.}
    \label{fig:qc-loop}
\end{figure}

Let $\gamma:[-\infty,+\infty] \to \Chat$ be a Jordan curve where $\g (-\infty) = \g (+\infty)$. We choose a family of uniformizing maps $H_t: \Chat \smallsetminus \gamma[-\infty,t] \to \m{C}\smallsetminus \m{R}_+$ such that $H_t(\gamma(t)) = 0$ and $H_t(\gamma(+\infty)) = \infty$ for each $t\in\m{R}$. Note that each $H_t$ is unique up to a real multiplicative factor. We fix a consistent normalization such that for every $s<t$, $H_t \circ H_s^{-1}(z) = z + o(z)$ as $z\to \infty$. Note that it suffices to fix the map $H_0$ and then normalize $H_t$ for $t\neq 0$ so that $H_t \circ H_0^{-1}$ has the correct asymptotic behavior.  
This is possible since if we write 
\begin{equation}\label{eq:slit-plane-chain}
    h_t = H_t \circ H_0^{-1} \quad \text{and} \quad f_t(z) = \sqrt{h_t(z^2)},
\end{equation}
then $f_t$ is a conformal map taking a some neighborhood of $\infty$ in $\m{H}$ to another neighborhood of $\infty$ in $\m{H}$. Schwarz reflection applied to $f_t$ in a neighborhood of $\infty$ shows that $f_t$ is holomorphic at $\infty$. Normalizing $h_t$ as above is equivalent to normalizing $f_t$ such that  $f_t(z) = z + O(1)$ as $z\to \infty$. 

\begin{df}\label{df:loop_driving}
Define the \emph{\textnormal(loop\textnormal) driving function} $t\mapsto \lambda_t$ and the \emph{\textnormal(loop\textnormal) capacity function} $t\mapsto a_t$ from the expansion
\begin{equation}\label{eq:loop-f_t-def}
    f_t(z) = z - \lambda_t + \frac{2a_t}{z} + O\bigg( \frac{1}{|z|^2} \bigg) \quad \text{as }z\to \infty.
\end{equation}
    We always have $\l_0 = a_0 = 0$ and $t \mapsto a_t$ is continuous and strictly increasing. If $a_t \to \pm\infty$ when $t \to \pm \infty$, then we can reparametrize $\g$ such that $a_t = t$ for every $t\in \m{R}$. In this case, we say that $\g$ is \emph{capacity-parametrized} by $\m R$.\footnote{\label{footnote:capacity}Similar to the chordal case, a general Jordan curve $\g$ may not have infinite capacity at either end of the root. A priori, $a_t \to T_\pm$ as $t \to \pm \infty$ where $T_+ \in (0,+\infty]$ and $T_- \in [-\infty, 0)$. In other words, $\g$ may be capacity-parametrized by a strict sub-interval of $\m R$.  
    However, we show in Lemma~\ref{lem:total-capacity} that if $I^L(\g) < +\infty$ (see Definition~\ref{def:loop-driving-fn}), then the capacity parametrization of $\g$ must be defined on all of $\m R$. Throughout Section~\ref{sec:WP}, we assume that $\g$ is capacity-parametrized by $\m R$.} We remind the reader that the capacity parametrization and the corresponding driving function depend on the choices of the orientation of $\g$, $\g (\pm\infty), \g (0)$, and $H_0$. 
\end{df} 

The reader may wonder about the different behaviors of the map $f_t$  depending on the sign of $t$, which seem to give a different meaning to the term ``capacity.'' This difference is not fundamental as the designation of the point of zero capacity on $\g$ is artificial. We shall view the capacity given by $f_t$ as the ``relative capacity'' with respect to our choice of the part $\g[-\infty,0]$. 
Precisely, it means the following.
\begin{lem}\label{lem:loop-time-change}
Suppose $\g$ is a  Jordan curve, capacity-parametrized by $\m R$, with driving function $t \mapsto \l_t$. 
Then, for every $s \in \m R$, $\sqrt {H_{s}\circ \gamma(\cdot+s)}$ defined on $\m R_+$ is a simple chord in $(\m H; 0, \infty)$ parametrized by capacity. 
Moreover, its chordal driving function is given by $\l_{\cdot + s} - \l_{s}$. 
\end{lem}
\begin{proof}
When $s = 0$,  $\eta (\cdot) = \sqrt{H_0\circ\gamma(\cdot)}$ is a simple chord in $\m{H}$ from $0$ to $\infty$. For $t >0$, the conformal map $h_t$ takes $H_0(\Chat\smallsetminus\gamma[-\infty,t])$ onto $\m{C}\smallsetminus \m{R}_+$.
Hence, $f_t$ is a conformal map from $\m{H}\smallsetminus\eta(0,t]$ onto $\m{H}$. Therefore, when $t>0$, \eqref{eq:loop-f_t-def} simply means that $2a_t = 2t$ is the half-plane capacity of $\eta[0,t]$ and $(\lambda_t)_{t\geq 0}$ is the chordal driving function of $\eta$ by \eqref{eq:f_eq}.

 For general $s \in \m R$, we note that $(f_{t+ s} \circ f_{s}^{-1})_{t\geq 0}$ is a centered Loewner chain associated with the curve $\sqrt{{H_{s} \circ \gamma(\cdot + s)}}$. Moreover, \eqref{eq:loop-f_t-def} implies
\begin{equation*}
    f_{t+s} \circ f_{s}^{-1}(z) = z - (\lambda_{t+s} - \lambda_{s}) + \frac{2(a_{t+s}-a_{s})}{z} + O\bigg(\frac{1}{z^2}\bigg) \quad \text{as } z\to\infty
\end{equation*}
for all $t\geq 0$. Thus, $t\mapsto \lambda_{t+s} - \lambda_{s}$ and $t\mapsto a_{t+s} - a_{s}$ are respectively the driving function and the capacity function corresponding to the chain $(f_{t+s} \circ f_{s}^{-1})_{t\geq 0}$. In particular, the assumption $a_t = t$ for all $t\in\m{R}$ means that the chain $(f_{t+s} \circ f_{s}^{-1})_{t\geq 0}$ is also in capacity parametrization.
\end{proof}

\begin{remark} The loop driving function generalizes the chordal Loewner driving function. If $\eta$ is a simple chord in $(\m H; 0, \infty)$ with driving function $\l: \m R_+ \to \m R$, then the Jordan curve $\g := \eta(\cdot)^2 \cup \m R_+$ with the same orientation as $\eta$ (from $0$ to $\infty$), root $\infty$, $\g (0) = 0$, and $H_0 (z) = z$, has the driving function $(\tilde \l_t)_{t\in\m{R}}$ where $\tilde \l_t  = \l_t$ if $t \ge 0$ and $\tilde \l_t = 0$ if $t \le 0$.
 \end{remark}

\begin{remark}\label{rem:mobius_inv_dr}
    Let $\g$ be a Jordan curve capacity-parametrized by $\m R$ using the conformal map $H_0$. If $A : \Chat \to \Chat$ is a M\"obius transformation, then the loop $t\mapsto \tilde \g(t) : =  A (\g(t))$ has the same driving function as $\g$ when we choose its capacity parametrization using the conformal map $\tilde H_0 = H_0\circ A^{-1}$. Moreover, the conformal maps corresponding to $H_t$ and $h_t$ for $\tilde\g$ are $\tilde H_t = H_t \circ A^{-1}$ and $\tilde h_t = \tilde H_t \circ \tilde H_0^{-1} =  H_t \circ H_0^{-1} = h_t$. Hence, the map $f_t$ remains unchanged, and so are the capacity and driving functions. 
\end{remark}

\begin{df} \label{def:loop-driving-fn}
The \emph{Loewner energy} of a Jordan curve $\g$ is 
$$I^L(\g) = \frac 12 \int_{-\infty}^{+\infty} \dot \lambda_t ^2 \,\dd t,$$
where $(\l_t)_{t\in \m R}$ is the driving function of $\g$   described above. See also Lemma~\ref{lem:total-capacity}. Theorem~\ref{thm:intro_equiv} shows that this energy does not depend on the parametrization of the curve (but we will not use this fact in our proof).
\end{df}
The next corollary is immediate after Lemma~\ref{lem:loop-time-change}.
\begin{cor} \label{cor:partial_loop}
For all $s<t$, the partial chordal Loewner energy of $\g(s,t]$ in the simply connected domain $\Chat \smallsetminus \g[-\infty,s]$ 
is given by 
    $$I^C_{\Chat \smallsetminus \g[-\infty,s]} (\g(s,t]) =\frac 12 \int_{s}^{t} \dot \lambda_r ^2 \,\dd r$$
    where the slit domain $\Chat \smallsetminus \g[-\infty,s]$ is always understood with the two marked prime ends being the two ends of $\g[-\infty,s]$.
\end{cor}
\begin{proof}
It suffices to notice that  $\sqrt{H_{s} (\cdot)}$ maps $\Chat \smallsetminus \g[-\infty,s]$ conformally onto $\m H$. Hence, the partial chordal Loewner energy of $\g(s,t]$ in $\Chat \smallsetminus \g[-\infty,s]$ equals that of $\sqrt{H_{s} (\g (s,t])}$ in $\m{H}$, which has the driving function $r \mapsto \l_{s +r} - \l_{s}$ defined for $r \in [0,t-s]$ by Lemma~\ref{lem:loop-time-change}.
\end{proof}

\subsection{Variation of Loewner energy for a part of the quasicircle}

We now consider deformations of a Jordan curve $\gamma$. Let $\mu \in L^\infty(\Chat)$ be a complex-valued function with compact support in $\Chat\smallsetminus\gamma$. For $\vare \in \m{R}$ with $\|\vare \mu\|_{\infty} < 1$, let $\omega^\emu:\Chat\to\Chat$ be a quasiconformal homeomorphism which satisfies the Beltrami equation 
\begin{equation*}
    \partial_{\bar z} \omega^\emu = (\vare\mu) \partial_z \omega^\emu.
\end{equation*}

Denote the deformation of $\gamma$ under the quasiconformal map $\omega^\emu$ as  
\begin{equation*}
    \gamma^\emu = \omega^\emu\circ \gamma.
\end{equation*}
Again we choose a family of uniformizing maps $H_t^\emu:\Chat \smallsetminus \gamma^\emu[-\infty,t] \to \m{C}\smallsetminus \m{R}_+$ with $H_t^\emu(\gamma^\emu(t)) = 0$ and $H_t^\emu(\gamma^\emu(+\infty)) = \infty$, normalized so that $H_t^\emu \circ \omega^\emu \circ H_t^{-1} (z) = z + o(z)$ as $z\to\infty$ for each $t\in \m R$. 

We define analogously the chains $(h_t^\emu)_{t\in\m{R}}$ and $(f_t^\emu)_{t\in\m{R}}$, the driving function $(\lambda_t^\emu)_{t\in\m{R}}$, and the capacity function $(a_t^\emu)_{t\in\m{R}}$. That is, 
\begin{equation*}
    h_t^\emu = H_t^\emu \circ (H_0^\emu)^{-1} \quad \text{and} \quad f_t^\emu(z) = \sqrt{h_t^\emu(z^2)}.
\end{equation*}
Then, by our choice of normalization, $h_t^\emu(z) = z + o(z)$ and $f_t^\emu(z) = z + O(1)$ as $z\to \infty$. 
We define $\lambda_t^\emu$ and $a_t^\emu$ from the expansion
\begin{equation*}
    f_t^\emu(z) = z -\lambda_t^\emu + \frac{2a_t^\emu}{z} + O\bigg(\frac{1}{z^2}\bigg) \quad \text{as } z\to \infty.
\end{equation*}

\begin{remark}
The map $\omega^\emu$, and hence the Jordan curve $\gamma^\emu$, is unique only up to a post-composition by some M\"obius transformation. The choice of $\omega^\emu$ does not affect our analysis, because the first step in it is always to apply the appropriately normalized uniformizing map $H_t^\emu$ from $\Chat \smallsetminus \gamma^\emu[-\infty,t]$ onto $\m{C}\smallsetminus \m{R}_+$.
\end{remark}

In this subsection, we translate Corollary~\ref{cor:chord-LoewnerE} into an analogous formula for the Weil--Peterson curve $\gamma$ and its deformation $\gamma^\emu$. The following is the main result.

\begin{prop}\label{prop:infinitesimal-loewner-energy}
    Suppose $s<t$ and  $I^C_{\Chat \smallsetminus \gamma[-\infty,s]}(\gamma(s,t])<\infty$. Then,
    \begin{equation}\label{eq:infinitesimal-loewner-energy}
        \frac{\partial}{\partial \vare} \bigg|_{\vare=0} I^C_{\Chat \smallsetminus\gamma^\emu[-\infty,s]}(\gamma^\emu[s,t]) =  -\frac{4}{\pi}\, \mathrm{Re} \int_{\m{C}\smallsetminus\gamma} \mu(z) (\Schwarzian H_t(z)  - \Schwarzian H_s(z))\,\dd^2 z .
    \end{equation}
\end{prop}

The following lemma is a straightforward calculation used in the proof of Proposition~\ref{prop:infinitesimal-loewner-energy}.

\begin{lem}\label{lem:d/dt-Sh_t}
    Suppose $\lambda_t$ is absolutely continuous. Then, for each $z \in H_0(\Chat\smallsetminus\gamma)$, the Schwarzian $\Schwarzian h_t(z)$ is absolutely continuous in $t$. Moreover, for almost every $t$,
    \begin{equation}
        \frac{\partial}{\partial t} \Schwarzian h_t(z) = -\frac{3 h_t'(z)^2}{4 h_t(z)^{5/2}} \dot \lambda_t. 
    \end{equation}
\end{lem}
\begin{proof}
    From the relation 
    \begin{equation*}
        \Schwarzian h_{t+u}(z) = \Schwarzian[h_{t+u}\circ h_t^{-1}](h_t(z)) \cdot h_t'(z)^2 + \Schwarzian h_t(z),
    \end{equation*}
    we deduce 
    \begin{equation*}
        \frac{\partial}{\partial t} \Schwarzian h_t(z) = \frac{\partial}{\partial u} \bigg|_{u=0} (\Schwarzian h_{t+u}(z) - \Schwarzian h_t(z)) = h_t'(z)^2 \cdot \frac{\partial}{\partial u} \bigg|_{u=0} \Schwarzian [h_{t+u} \circ h_t^{-1}](h_t(z)).
    \end{equation*}
    Hence, it suffices to show that 
    \begin{equation}\label{eq:d/dt-Sh_t-difference}
        \frac{\partial}{\partial u}\bigg|_{u=0} \Schwarzian [h_{t+u} \circ h_t^{-1}](h_t(z)) = -\frac{3\dot \lambda_t}{4h_t(z)^{5/2}}.
    \end{equation}

    To see this, note that $\tilde f_u:= f_{t+u} \circ f_t^{-1}$ solves the Loewner equation   (see Lemma~\ref{lem:loop-time-change})
    \begin{equation*}
        \partial_u \tilde f_u(z) = \frac{2}{\tilde f_u(z)} - \dot \lambda_{t+u}.
    \end{equation*}
    Since $\tilde h_u(z):= h_{t+u} \circ h_t^{-1}(z) = \tilde f_u(\sqrt{z})^2 $, we have 
    \begin{equation*}
        \partial_u \tilde h_u(z) = 2\tilde f_u(\sqrt{z})\partial_u \tilde f_u(\sqrt{z}) = 4 - 2 \dot \lambda_{t+u} \tilde h_u(z)^{1/2}.
    \end{equation*}
    Then, because $\tilde h_0(z) = z$,
    \begin{equation}\label{eq:d/dt-Sh_t-t=0}
        \frac{\partial (\Schwarzian \tilde h_u(z))}{\partial u}\bigg|_{u=0} = \big(\partial_u \tilde h_u(z)\big|_{u=0}\big)''' = (4- 2\dot \lambda_t z^{1/2})''' = -\frac{3\dot \lambda_t}{4z^{5/2}}.
    \end{equation}
    Replacing $z$ in \eqref{eq:d/dt-Sh_t-t=0} with $h_t(z)$, we obtain \eqref{eq:d/dt-Sh_t-difference}. This completes the proof.
\end{proof}

\begin{proof}[Proof of Proposition~\ref{prop:infinitesimal-loewner-energy}]
    Let us first consider the case $s=0$. Letting $\eta(t) = \sqrt{H_0 \circ \gamma(t)}$ and $\eta^\vare(t) = \sqrt{H_0^\emu \circ \gamma^\emu(t)} = \sqrt{H_0^\emu \circ \omega^\emu \circ \gamma(t)}$, we have $\eta^\vare = \psi^\vare \circ \eta$ where $\psi^\vare(z) = \sqrt{(H_0^\emu \circ \omega^\emu \circ H_0^{-1})(z^2)}$. Let $\enu$ be the Beltrami coefficient corresponding to the quasiconformal map $\psi^\vare:\m{H}\to\m{H}$. Let $\vare \mu_0$ denote the Beltrami coefficient of $H_0^\emu \circ \omega^\emu \circ H_0^{-1}$. Then, $\nu(\zeta) = \mu_0(\zeta^2)(|\zeta|^2/\zeta^2)$.

    Substituting this $\nu$ into \eqref{eq:chordal-energy-variation} and letting $\zeta = \sqrt{z}$, since $f_t(\zeta) = \sqrt{h_t(z)}$ and $f_t'(\zeta)/\zeta = h_t'(z) /\sqrt{h_t(z)}$, we get
    \begin{equation*}
    \begin{aligned}
        \frac{\partial}{\partial \vare}\bigg|_{\vare=0} I^C(\eta^\vare(0,T])  &= \frac{12}{\pi} \,\mathrm{Re}  \int_{\m{H}} \nu(\zeta) \int_0^T \dot \lambda_t \frac{f_t'(\zeta)^2}{f_t(\zeta)^3} \dd t \, \dd^2 \zeta   \\
        &= \frac{3}{\pi} \,\mathrm{Re}  \int_{\m{C}\smallsetminus \m{R}_+} \mu_0(z) \int_0^T \dot \lambda_t \frac{h_t'(z)^2}{h_t(z)^{5/2}} \dd t \, \dd^2 z.
    \end{aligned}
    \end{equation*}
    Applying Lemma~\ref{lem:d/dt-Sh_t}, we have 
    \begin{equation*}
    \begin{aligned}
        \frac{\partial}{\partial \vare}\bigg|_{\vare=0} I^C(\eta^\vare(0,T])  &= - \frac{4}{\pi} \, \mathrm{Re} \int_{\m{C}\smallsetminus\m{R}_+}  \mu_0(z) \int_0^T \frac{\partial (\Schwarzian h_t(z))}{\partial t}\, \dd t \, \dd^2 z  \\
        &= - \frac{4}{\pi}\, \mathrm{Re} \int_{\m{C}\smallsetminus\m{R}_+} \mu_0(z) \Schwarzian h_T(z)\, \dd^2 z .
    \end{aligned}
    \end{equation*}
    Recalling our definition of $\vare \mu_0$, we have $\mu_0(H_0(z)) = \mu(z) (H_0'(z)^2/|H_0'(z)|^2)$. Moreover, from $H_T = h_T \circ H_0$, we have 
    $$\Schwarzian h_T(H_0(z)) \cdot H_0'(z)^2 = \Schwarzian H_T(z) - \Schwarzian H_0(z).$$ Hence,
    \begin{equation*}
        \int_{\m{C}\smallsetminus\m{R}_+} \vare \mu_0(z) \Schwarzian h_T(z)\, \dd^2 z = \int_{\m{C}\smallsetminus \gamma} \vare \mu(z) (\Schwarzian H_T(z) - \Schwarzian H_0(z))\, \dd^2 z.
    \end{equation*}
    Therefore, the case $s=0$ holds. In fact, this implies \eqref{eq:infinitesimal-loewner-energy} for any $s\in \m{R}$ because the parametrization of $\gamma$ is arbitrary up to translations as discussed in Lemma~\ref{lem:loop-time-change}.
\end{proof}

\subsection{Variation of the loop Loewner energy}

The goal of this subsection is to prove Theorem~\ref{thm:intro_energy-derivative-schwarzian}.
Let $H_{+\infty} :\Chat \smallsetminus \gamma \to \m{C}\smallsetminus\m{R}$ be any conformal map which maps $\O \to \m H$ and $\O^* \to \m H^*$.
Note that the map $H_{+\infty}$ restricted to $\O$ coincides with $f^{-1}$ (as in Theorem~\ref{thm:intro_energy-derivative-schwarzian}) post-composed by a M\"obius transformation, so $\mc S H_{+\infty}|_{\O} = \mc S [f^{-1}]$. Similarly, $\mc S H_{+\infty}|_{\O^*} = \mc S [g^{-1}]$.
In view of Corollary~\ref{cor:partial_loop} and  Proposition~\ref{prop:infinitesimal-loewner-energy}, it suffices to show that as $s\to -\infty$ and $t\to+\infty$, we have
\begin{equation}\label{eq:energy-limit}
    \int_{\m{C}\smallsetminus\gamma} \mu(z) (\Schwarzian H_t(z) - \Schwarzian H_s(z))\, \dd^2 z \to \int_{\m{C}\smallsetminus\gamma} \mu(z) \Schwarzian H_{+\infty} (z)\, \dd^2 z,
\end{equation}
and  
\begin{equation}\label{eq:energy-limit-2}
  \frac{\dd}{\dd \vare} \bigg|_{\vare=0}  \int_s^t \left|\frac{\dd \lambda_r^\emu}{\dd a_r^\emu}\right|^2 \dd a_r^\emu \to  \frac{\dd}{\dd \vare}  \bigg|_{\vare=0} \int_{-\infty}^{+\infty} \left|\frac{\dd \lambda_r^\emu}{\dd a_r^\emu}\right|^2 \dd a_r^\emu .
\end{equation}

For this, we need a few lemmas.
\begin{lem}\label{lem:Schwarzian-negative-limit}
    Suppose $\gamma(\pm\infty) = \infty$. Then, $\Schwarzian H_s \to 0$ locally uniformly as $s\to -\infty$.
\end{lem}
\begin{proof}
    Given any $R>0$, there exists a large negative $s_R$ such that for $s\leq s_R$, we have $\gamma[-\infty,s] \cap \{z: |z| < R\} = \varnothing$. Then, $H_s$ is conformal on $R\m{D}$. By the Nehari bound, 
    $$|\Schwarzian H_s(z)| \leq \frac{6}{\left(R(1-|z|^2/R^2)\right)^2}$$ for every $z\in R\m{D}$. The right-hand side of the inequality tends uniformly to 0 as $R\to +\infty$ on any compact subset of $\m{C}$.
\end{proof}

\begin{lem}\label{lem:Schwarzian-positive-limit}
    Suppose $\gamma(\pm \infty) = \infty$. Then, $\Schwarzian H_t \to \Schwarzian H_{+\infty}$ locally uniformly on $\m{C} \smallsetminus\gamma$ as $t\to+\infty$.
\end{lem}
\begin{proof}
    Choose either component of $\m{C} \smallsetminus \gamma$ and call it $U_{+\infty}$. 
    Let us denote by $\gamma^U(s)$ the prime end of $\gamma(s)$ as viewed from $U_{+\infty}$.
    
    Let $\gamma_t = \gamma[-\infty,t]$ and denote by $\gamma_t^U := \bigcup_{s\in(-\infty,t)}\gamma^U(s)$ the prime ends of $\gamma_t$ accessible from $U_{+\infty}$. Let $\Gamma_t$ be the hyperbolic geodesic in $\m{C} \smallsetminus \gamma_t$ connecting $\gamma(t)$ and $\gamma(+\infty)$. 
    Let $U_t$ be the component of $\m{C} \smallsetminus (\gamma_t \cup \Gamma_t)$ such that the prime ends of $\gamma_t$ as viewed from $U_t$ comprise $\gamma^U_t$. Observe that if $\mathrm{hm}(z,D;\cdot)$ is the harmonic measure on the domain $D$ as viewed from $z\in D$, then $z\in U_t$ if and only if $\mathrm{hm}(z,\m{C}\smallsetminus\gamma_t;\gamma_t^U) > 1/2$.

    We claim that if $z \in U_{+\infty}$, then $z \in \bigcup_{T\geq 0} \bigcap _{t\geq T} U_t$. 
    Suppose $z\in U_{+\infty}$. Choose a sufficiently small constant $a\in(0,1)$ such that $\mathrm{hm}(0,\m{D}\smallsetminus[a,1], [a,1]) > 1/2$. Since $\gamma(\pm\infty)=\infty$, for all sufficiently large $t>0$, we can find $R>0$ such that $\gamma_t \cap B_{aR}(z) \neq \varnothing$ but $\gamma(t,+\infty) \cap B_R(z) = \varnothing$. Then, 
    \begin{equation*}
        \mathrm{hm}(z,\m{C}\smallsetminus \gamma_t;\gamma_t^U) \geq \mathrm{hm}(z,B_R(z) \smallsetminus \gamma_t; \gamma_t^U) = \mathrm{hm}(z,B_R(z) \smallsetminus \gamma_t; \gamma_t).
    \end{equation*}
    By the Beurling projection theorem, 
    \begin{equation*}
        \mathrm{hm}(z,B_R(z) \smallsetminus \gamma_t, \gamma_t) \geq \mathrm{hm}(0, R\m{D} \smallsetminus [aR,R]; [aR,R]) = \mathrm{hm}(0,\m{D} \smallsetminus[a,1]; [a,1]) >1/2.
    \end{equation*}
    This completes the proof of the claim. 
    
    Note that $H_t$ is a conformal map which sends $\m{C} \smallsetminus (\gamma_t \cup \Gamma_t)$ onto $\m{C}\smallsetminus \m{R}$. 
    Then, by the Carath\'eodory kernel theorem, $H_t$ post-composed with an appropriate M\"obius transformation converges locally uniformly on $U_{+\infty} $ to $H_{+\infty}$ as $t\to\infty$. Consequently, $\Schwarzian H_t \to \Schwarzian H_{+\infty}$ locally uniformly on $U_{+\infty}$ as $t\to\infty$. An analogous argument applies to the other component of $\Chat\smallsetminus\gamma$.
\end{proof}

\begin{proof}[Proof of Theorem~\ref{thm:intro_energy-derivative-schwarzian}]
    Given $\gamma(\pm \infty) = \infty$, the limit \eqref{eq:energy-limit} follows by applying Lemmas~\ref{lem:Schwarzian-negative-limit} and \ref{lem:Schwarzian-positive-limit}. Otherwise, choose a M\"obius map $A:\Chat \to \Chat$ so that $A(\infty) = \gamma(\pm \infty)$. If $H_{+\infty} :\Chat \smallsetminus \gamma \to \m{C}\smallsetminus\m{R}$ is a conformal map as in the theorem statement, then $H_{+\infty} \circ A$ is a conformal map from $\Chat \smallsetminus \gamma$ onto $\m{C}\smallsetminus\m{R}$ with $\Schwarzian[H\circ A] = \Schwarzian H \cdot (A')^2$. The pullback $A^*\mu$ of $\mu$ under $\varphi$ satisfies $\mu(A(z)) = (A^*\mu)(z) (A'(z))^2/|A'(z)|^2$. Letting $\zeta = A(z)$, we have
    \begin{equation*}
        \int_{\Chat \smallsetminus \gamma} \mu(\zeta) \Schwarzian H_{+\infty} (\zeta) \,\dd^2 \zeta = \int_{\Chat \smallsetminus A^{-1}\circ\gamma} A^*\mu(z) \Schwarzian[H\circ A](z)\, \dd^2 z,
    \end{equation*}
    so we can consider the curve $A^{-1}\circ\gamma$ instead.
    
    To show \eqref{eq:energy-limit-2}, it suffices to prove that we can switch between the integral over $t\in (-\infty, +\infty)$ and the derivative in $\vare$. To this end, we prove that the following integral is absolutely convergent:
    \begin{equation}\label{eq:integrability-condition} \int_{-\infty}^{+\infty} \left| \frac{\partial}{\partial \vare} \bigg|_{\vare=0}  \frac{(\dot \lambda_t^\emu)^2}{\dot a_t^\emu}\right| \dd t < +\infty. \end{equation}
Recall from the proof of Corollary~\ref{cor:chord-LoewnerE} that for $t\geq 0$,
    \begin{equation} \label{eq:loop-energy-dt}
        \frac{\partial}{\partial \vare} \bigg|_{\vare=0}  \frac{(\dot \lambda_t^\emu)^2}{\dot a_t^\emu}= \frac{24 \dot\l_t}{\pi} \Re \int_{\m{H}} \nu(z)  \frac{f_t'(z)^2}{f_t(z)^3} \, \dd^2 z 
    \end{equation}
    where $\nu$ is the push-forward of the Beltrami differential $\mu$ under the map $\sqrt{H_0}$ as displayed in Figure~\ref{fig:qc-loop}. Using Lemma~\ref{lem:loop-time-change} and the composition rule \eqref{eq:nu_t-coord-change} for Beltrami differentials, it is straightforward to check that \eqref{eq:loop-energy-dt} is true for all $t\in \m{R}$.
    From Lemma~\ref{lem:appendix}, the assumption that $\int_{-\infty}^{+\infty} |\dot\l_t|^2\,\dd t <\infty$, and the Cauchy--Schwarz inequality, it follows that $\int_{-\infty}^{+\infty} |\dot \l_t \frac{f_t'(z)^2}{f_t(z)^3}|\,\dd t$ is finite and locally uniform in $z$.
    This implies \eqref{eq:integrability-condition} since $\nu$ is compactly supported and $\norm{\nu}_\infty < \infty$.
    Therefore, for any $-\infty \leq s \leq t \leq +\infty$,
    \begin{equation}
        \frac{\partial}{\partial \vare}  \bigg|_{\vare=0} \int_s^t \left|\frac{\dd \lambda_r^\emu}{\dd a_r^\emu}\right|^2 \dd a_r^\emu = \frac{\partial}{\partial \vare}  \bigg|_{\vare=0} \int_s^t \frac{(\dot \l_r^\emu)^2}{\dot a_r^\emu} \dd r = \int_s^t \frac{\partial}{\partial \vare}  \bigg|_{\vare=0} \frac{(\dot \l_r^\emu)^2}{\dot a_r^\emu} \dd r.
    \end{equation}
    Letting $s\to -\infty$ and $t\to +\infty$, by the dominated convergence theorem and \eqref{eq:integrability-condition}, we conclude \eqref{eq:energy-limit-2}. This proves Theorem~\ref{thm:intro_energy-derivative-schwarzian} as explained at the beginning of the subsection.
\end{proof}

\section{Remarks and open questions}\label{sec:comments}

While Loewner chains have been studied extensively due to their applications in the study of Schlicht functions and, more recently, in that of Schramm--Loewner evolution (SLE), their infinitesimal variations seem to have been overlooked. Our investigation of this topic is motivated by an effort to understand the large deviations of SLE at a deeper level. Roughly speaking, an SLE curve is a non-self-intersecting curve whose Loewner chain is driven by a constant multiple of Brownian motion. We refer the reader to
the textbooks~\cite{WW_St_Flour,Law05,Kemppainen_book,Beliaevs_book} for detailed introductions to SLE and its applications.

Theorem~\ref{thm:intro_equiv} implies that the Loewner energy is at the same time the action functional of an SLE loop \cite{W1,carfagnini_wang} and the K\"ahler potential on the Weil--Petersson Teichm\"uller space $T_0(1)$ \cite{W2,TT06}, thus building a bridge between two fundamentally different perspectives on the geometry of the space of Jordan curves in $\Chat$. On the SLE side, Jordan curves are viewed as dynamically growing slits, which are naturally described through the language of Loewner chains. The link from SLE to Loewner energy comes from stochastic analysis and the fact that the action functional of a Brownian motion is its Dirichlet energy. On the K\"ahler geometry side, there are no dynamics. Instead, all geometric structures are expressed infinitesimally on the tangent spaces of $T_0(1)$. 
The reason behind the fact that the action functional of SLE coincides with the K\"ahler potential of the unique homogeneous K\"ahler metric on $T_0(1)$ remains a mystery. See \cite{W_AMS} for an expository article on this link. 

Our motivation lies in building tools that can be used to reconcile these two distinct viewpoints. The current work serves as the first step in one of the possible directions to this end by elucidating which infinitesimal variations of the Loewner driving function correspond to those on $T_0(1)$. 
Natural questions going forward are how the complex structure, the symplectic form,  the Weil--Petersson metric, and the group structure on $T_0(1)$ are encoded in the Loewner driving function. Through these identifications, there is hope to build a more robust connection between random conformal geometry and Teichm\"uller theory and shed new light on their relationship.

There are yet other possible avenues in the study of quasiconformal deformations of SLE. For example, there is an interesting open question \cite[Conj.\ 7.1]{binder2023conformal} to identify the conformal dimension of SLE, which is the minimal Hausdorff dimension of the image of an SLE curve under quasiconformal mappings. The relevance of our work in this direction is that some of our results are sufficiently general to be applied in this context. For instance, the variational formula for the Loewner driving function (Theorem~\ref{thm:intro_driving_function}) does not assume any regularity on the driving function, so it can be applied even when the driving function is a Brownian motion. 
On the other hand, there is room for improvement in our results. The most obvious limitation is that we require the support of the Beltrami differential to be away from the curve so as not to deal with improper integrals. The deformations we considered in this work are thus conformal in a neighborhood of the curve, and in this regard, we are still in the same setup as in the conformal restriction of SLE considered in \cite{LSW_CR_chordal}. 
To understand general quasiconformal deformations of SLE, we need to allow the supports of Beltrami differentials to touch the curve. We think this is an interesting question that may require taking into account the stochastic nature of SLE.  

\newpage
\appendix
\section{Large time behavior of finite-energy Loewner chains}

A priori, a Jordan curve may have finite total capacity. However, we now show that if a Jordan curve has finite Loewner energy, then its capacity must be infinite at both its initial and terminal parts. 

\begin{lem}\label{lem:total-capacity}
    Let $\g:[T_-,T_+]\to\Chat$ be a Jordan curve parametrized by capacity, where $T_- \in [-\infty,0)$, $T_+ \in (0,+\infty]$ and $\g(T_-) = \g(T_+)$ is the root. More precisely, this means that there is a conformal map $H_0: \Chat \smallsetminus \g[T_-,0] \to \m{C} \smallsetminus \m{R}_+$ such that for $a_t$ defined as in Definition~\ref{df:loop_driving}, we have $a_t = t$ for all $t\in [T_-,T_+]$. If
    $ I^L(\g) := \frac{1}{2}\int_{T_-}^{T_+} |\dot \lambda_t|^2\, \dd t< \infty$,  then $T_- = -\infty$ and $T_+ = +\infty$.
\end{lem}
Consequently, the capacity parametrization of a finite-energy Jordan curve must take all of $\m R$ as its domain, which justifies the formula in Definition~\ref{def:loop-driving-fn}.
\begin{proof}
    If the Loewner energy of the chord $\g(0,T_+)$ traversing the domain $\Chat \smallsetminus \g[T_-,0]$ is finite, then its total capacity $T_+$ must be infinite by \cite[Thm.\ 2.4]{sle_ld_survey}. By definition, this chordal Loewner energy is bounded above by the loop Loewner energy $I^L(\g)$.

    Let us now consider $T_-$. By Carath\'eodory's theorem, the conformal map $\sqrt{H_0}$ from $\Chat \smallsetminus \g[T_-,0]$ onto $\m{H}$ extends continuously to a bijection between the prime ends  of $\Chat \smallsetminus\g[T_-,0]$ and $\m{R} \cup \{\infty\}$. Note that $\sqrt{H_0(\g(0))}=0$ and $\sqrt{H_0(\g(T_-))} = \infty$, and for each $t\in (T_-,0)$, one of the prime ends at $\g(t)$ is mapped to a point $x_t \in \m{R}_+$. Let $f_t$ and $\l_t$ be defined as in \eqref{eq:slit-plane-chain} and \eqref{eq:loop-f_t-def}. By Lemma~\ref{lem:loop-time-change}, the Loewner equation $\partial_t f_t(z) = 2/f_t(z) - \dot \lambda_t$ holds for $t\in (T_-,0)$ as well. This implies
    \begin{equation*}
        \frac{\partial}{\partial t} (\Re f_t(z))^2 = \frac{4(\Re f_t(z))^2}{|f_t(z)|^2} - 2\dot \l_t (\Re f_t(z)) .
    \end{equation*} 
    (See \eqref{eq:u_t-and-v_t} below.) For $T\in(T_-,0)$, we deduce from above using Schwarz reflection that 
    \begin{equation}\label{eq:real-axis-loewner}
        \frac{\dd (f_t(x_T))^2}{\dd t} = 4 - 2\dot \lambda_t f_t(x_T), \quad t \in (T,0].
    \end{equation}
    
    Assume, for the sake of contradiction, that $I^L(\g)< +\infty$ and $T_- > -\infty$. Since $f_T(x_T) = 0$, we deduce from \eqref{eq:real-axis-loewner} using the Cauchy--Schwarz inequality that
    \begin{equation*}
    \begin{aligned}
        \sup_{t\in [T,0]}|f_t(x_T)|^2 &\leq 4|T| + 2 \bigg( \int_T^0 \dot |\lambda_t|^2\, \dd t\bigg)^{\frac{1}{2}} \bigg( \int_T^0 |f_t|^2 \, \dd t \bigg)^{\frac{1}{2}} \\
        &\leq 4|T_-| + 2 \sqrt{I^L(\g) |T_-|}  \bigg( \sup_{t\in [T,0]}|f_t(x_T)|^2 \bigg)^{\frac{1}{2}} .
    \end{aligned}
    \end{equation*}
    Completing the square in the above inequality, we obtain
    \begin{equation}\label{eq:negative-capacity-ineq}
        |x_T|^2 \leq \sup_{t\in [T,0]} |f_t(x_T)|^2 \leq  |T_-| \Big(2+\sqrt{I^L(\g)}\Big)^2  . 
    \end{equation}
    By assumption, the right-hand side of \eqref{eq:negative-capacity-ineq} is a finite quantity independent of $T \in (T_-,0)$. This is a contradiction since $t\mapsto x_t$ maps $[T_-,0]$ onto $[0,+\infty]$. 
\end{proof}

The following lemma is used in the proof of Theorem~\ref{thm:intro_energy-derivative-schwarzian} to switch the order between the derivative and the integral in \eqref{eq:energy-limit-2}. 
Note that if $\lambda_t = 0$ for all $t$, then $f_t(z) = \sqrt{z^2 + 4t}$. In particular, $f_t(z)^2/(4t) \to 1$ and $2|t|^{1/2}|f_t'(z)| \to 1$ as $t \to \pm \infty$. 

\begin{lem}\label{lem:appendix}
    Let $\gamma$ be a Jordan curve with $I^L(\gamma) < \infty$. Let $H_t$ and $f_t$ be the uniformizing maps defined in Section~\ref{sec:driving}. Then, as $t\to \pm\infty$, $f_t(z)^2/(4t)\to 1$ and $|f_t'(z)| = |t|^{-1/2 + o(1)}$ locally uniformly for $z \in H_0(\Chat \smallsetminus \gamma)$. 
\end{lem}
\begin{proof}
    The starting point of this proof is similar to that of Lemma~\ref{lem:total-capacity}. Let us denote $f_t(z) = u_t + \ii v_t$, where $u_t \in \m{R}$ and $v_t>0$.
    The real and imaginary parts of the Loewner equation $\partial_t f_t(z) = 2/f_t(z) - \dot \lambda_t$ correspond to
    \begin{equation}\label{eq:u_t-and-v_t}
        \dot u_t = \frac{2u_t}{u_t^2 + v_t^2} - \dot \lambda_t \quad \text{and} \quad \dot v_t = -\frac{2v_t}{u_t^2 + v_t^2}.
    \end{equation}
    
    Let us consider the $t\to +\infty$ limit first. Given $\vare>0$, since $\g$ has finite Loewner energy, we can choose a large $T_0$ so that $\int_{T_0}^{+\infty} \dot \lambda_t^2\, \dd t < \vare$. Since
    \begin{equation*}
        \frac{\dd u_t^2}{\dd t} = \frac{4 u_t^2}{u_t^2 + v_t^2} - 2\dot \lambda_t u_t,
    \end{equation*}
    for $T_0\leq t\leq T$, we have
    \begin{equation}\label{eq:u_t^2-ineq}
    \begin{aligned}
        u_t^2 &\leq u_{T_0}^2 + 4\int_{T_0}^T \frac{u_t^2}{u_t^2 + v_t^2}\,\dd t + 2\int_{T_0}^T |\dot \lambda_t u_t|\,\dd t \\
        &\leq u_{T_0}^2 + 4(T-T_0) + 2 \bigg( \int_{T_0}^T \dot \lambda_t^2\, \dd t\bigg)^{1/2} \bigg( \int_{T_0}^T u_t^2 \, \dd t \bigg)^{1/2} .
    \end{aligned}
    \end{equation}
    Hence,
    \begin{equation*}
        \sup_{t\in [T_0,T]} u_t^2 \leq u_{T_0}^2 + 4(T-T_0) + 2\bigg( \vare (T-T_0) \sup_{t\in [T_0,T]} u_t^2 \bigg)^{1/2} .
    \end{equation*}
    Completing the square, we obtain
    \begin{equation*}
        \bigg( \sup_{t\in [T_0,T]} u_t^2 \bigg)^{1/2} \leq \big( u_{T_0}^2 + (4+\vare)(T-T_0) \big)^{1/2} + \big(\vare (T-T_0)\big)^{1/2}.
    \end{equation*}
    Then, there exists a $T_1 \geq T_0$ such that for all $T\geq T_1$, 
    \begin{equation}\label{eq:sup-u_t^2}
        \sup_{t\in [T_0,T]} u_t^2 \leq u_{T_0}^2 + 4(1+\vare) (T-T_0).
    \end{equation}

    Now, from
    \begin{equation*}
       \frac{\partial }{\partial t}\mathrm{Re}(f_t(z)^2) = \frac{\dd (u_t^2 - v_t^2)}{\dd t} = 4 - 2\dot \lambda_t u_t,
    \end{equation*}
    we have
    \begin{equation*}
    \begin{aligned}
        \big|\mathrm{Re} (f_T(z)^2)  - 4T\big| &\leq |\mathrm{Re}(f_{T_0}(z)^2)| + 4T_0 + 2\int_{T_0}^T |\dot\lambda_t u_t|\,\dd t \\
        &\leq |\mathrm{Re}(f_{T_0}(z)^2)| + 4T_0 + 2\bigg( \vare (T-T_0) \sup_{t\in [T_0,T]} u_t^2 \bigg)^{1/2}.
    \end{aligned}
    \end{equation*}
    Substituting \eqref{eq:sup-u_t^2}, we can find a $T_2 \geq T_1$ such that for all $T\geq T_2$,
    \begin{equation*}
        \big|\mathrm{Re} (f_T(z)^2)  - 4T\big| \leq 5\sqrt{\vare} T.
    \end{equation*}
    Since the choice of $\vare$ was arbitrary, we conclude $\mathrm{Re}(f_t(z)^2)/(4t) \to 1$ as $t\to +\infty$. 
    
    As for $\mathrm{Im}(f_t(z)^2) = 2u_tv_t$, note that \eqref{eq:u_t-and-v_t} implies $v_t$ is monotonically decreasing. Hence, \eqref{eq:sup-u_t^2} implies $\mathrm{Im}(f_t(z)^2)/t \to 0$ as $t\to +\infty$. Combining the limits of the real and imaginary parts, we obtain $f_t(z)^2/(4t) \to 1$ as $t\to +\infty$.
    Moreover, since $f_{T_0}(z) = u_{T_0} + \ii v_{T_0}$ depends continuously on $z$ whereas $T_0$ was chosen independently of $z$, the limit we proved converges uniformly on each compact subset of $H_0(\Chat \smallsetminus \gamma)$.

    Let us now consider the $t\to -\infty$ limit. This time, since the Loewner energy of $\g$ is finite, we can find a large negative $\tilde T_0$ such that $\int_{-\infty}^{\tilde T_0} \dot \lambda_t^2 \,\dd t < \vare$. Hence, \eqref{eq:u_t^2-ineq} implies that there exists a $\tilde T_1 \leq \tilde T_0$ such that for all $T \leq \tilde T_1$,
    \begin{equation*}
        \sup_{t\in [T, \tilde T_1]} u_t^2 \leq u_{\tilde T_0}^2 + 4(1+\vare) |T-T_0|.
    \end{equation*}
    Again, we have 
    \begin{equation*}
        \big|\mathrm{Re} (f_T(z)^2)  - 4T\big| \leq |\mathrm{Re}(f_{T_0}(z)^2)| + 4|\tilde T_0| + 2\bigg( \vare |T-\tilde T_0| \sup_{t\in [T,\tilde T_0]} u_t^2 \bigg)^{1/2},
    \end{equation*}
    and choosing $\vare$ to be arbitrarily small, we obtain $\mathrm{Re}(f_t(z)^2)/(4t) \to 1$ as $t\to -\infty$.   

    For the imaginary part of $f_t(z)^2$, we consider 
    \begin{equation*}
        \frac{\partial }{\partial t} \mathrm{Im}(f_t(z)^2) = \frac{\dd (2u_tv_t)}{\dd t} = -2\dot \lambda_t v_t.
    \end{equation*}
    Since $\mathrm{Re}(f_t(z)^2) = u_t^2 - v_t^2$, we have that as $T\to -\infty$, 
    \begin{equation*}
        \sup_{t\in [T,\tilde T_0]} v_t^2 \leq \sup_{t\in [T,\tilde T_0]} |\mathrm{Re}(f_t(z)^2)| + \sup_{t\in [T,\tilde T_0]} u_t^2 \leq (8+o(1)) |T|.
    \end{equation*}
    Hence, using the Cauchy--Schwarz inequality as above, we have
    \begin{equation*}
        \big| \mathrm{Im}(f_T(z)^2) \big| \leq \big|\mathrm{Im} (f_{\tilde T_0}(z)^2)\big| + 2\int_T^{\tilde T_0} |\dot \lambda_t v_t|\,\dd t = o(T)
    \end{equation*}
    as $T\to -\infty$. Therefore, $f_t(z)^2 / (4t) \to 1$ as $t\to -\infty$. Again, this limit converges uniformly on compact subsets of $\m{H}$ since $f_{\tilde T_0}(z)$ depends continuously on $z$ and $\tilde T_0$ can be chosen independently of $z$ on a compact set.

    Finally, to estimate $|f_t'(z)|$, consider the equation
    \begin{equation*}
        \frac{\partial}{\partial t}\log |f_t'(z)| = \mathrm{Re}\bigg( \frac{\partial_t f_t'(z)}{f_t'(z)} \bigg) = -\mathrm{Re}\frac{2}{f_t(z)^2}.
    \end{equation*}
    As $t\to \pm \infty$, the right-hand side behaves as $(-1/2 + o(1))t^{-1}$. We thus obtain $|f_t'(z)| = |t|^{-1/2+o(1)}$ as claimed. 
\end{proof}
\newpage
\bibliographystyle{abbrv}
\bibliography{ref}

\begin{thebibliography}{10}

\bibitem{AhlforsBers}
L.~Ahlfors and L.~Bers.
\newblock Riemann's mapping theorem for variable metrics.
\newblock {\em Ann. of Math. (2)}, 72:385--404, 1960.

\bibitem{Beliaevs_book}
D.~Beliaev.
\newblock {\em Conformal maps and geometry}.
\newblock Advanced Textbooks in Mathematics. World Scientific Publishing Co.
  Pte. Ltd., Hackensack, NJ, 2020.

\bibitem{binder2023conformal}
I.~Binder, H.~Hakobyan, and W.-B. Li.
\newblock Conformal dimension of the brownian graph, 2023.

\bibitem{carfagnini_wang}
M.~Carfagnini and Y.~Wang.
\newblock Onsager--{M}achlup functional for $\text{SLE}_{\kappa}$ loop
  measures, 2023.

\bibitem{DeBranges1985}
L.~de~Branges.
\newblock A proof of the {B}ieberbach conjecture.
\newblock {\em Acta Math.}, 154(1-2):137--152, 1985.

\bibitem{FrizShekhar15}
P.~K. Friz and A.~Shekhar.
\newblock On the existence of {SLE} trace: finite energy drivers and
  non-constant {$\kappa $}.
\newblock {\em Probab. Theory Related Fields}, 169(1-2):353--376, 2017.

\bibitem{Kemppainen_book}
A.~Kemppainen.
\newblock {\em Schramm-{L}oewner evolution}, volume~24 of {\em SpringerBriefs
  in Mathematical Physics}.
\newblock Springer, Cham, 2017.

\bibitem{LLN_capacity}
S.~Lalley, G.~Lawler, and H.~Narayanan.
\newblock Geometric interpretation of half-plane capacity.
\newblock {\em Electron. Commun. Probab.}, 14:566--571, 2009.

\bibitem{LSW_CR_chordal}
G.~Lawler, O.~Schramm, and W.~Werner.
\newblock Conformal restriction: the chordal case.
\newblock {\em J. Amer. Math. Soc.}, 16(4):917--955, 2003.

\bibitem{Law05}
G.~F. Lawler.
\newblock {\em Conformally invariant processes in the plane}, volume 114 of
  {\em Mathematical Surveys and Monographs}.
\newblock American Mathematical Society, Providence, RI, 2005.

\bibitem{LW2004loupsoup}
G.~F. Lawler and W.~Werner.
\newblock The {B}rownian loop soup.
\newblock {\em Probab. Theory Related Fields}, 128(4):565--588, 2004.

\bibitem{Lind_Tran}
J.~Lind and H.~Tran.
\newblock Regularity of {L}oewner curves.
\newblock {\em Indiana Univ. Math. J.}, 65(5):1675--1712, 2016.

\bibitem{Lind_sharp}
J.~R. Lind.
\newblock A sharp condition for the {L}oewner equation to generate slits.
\newblock {\em Ann. Acad. Sci. Fenn. Math.}, 30(1):143--158, 2005.

\bibitem{Loewner1923}
K.~Loewner.
\newblock Untersuchungen \"{u}ber schlichte konforme {A}bbildungen des
  {E}inheitskreises. {I}.
\newblock {\em Math. Ann.}, 89(1-2):103--121, 1923.

\bibitem{Marshall_Rohde}
D.~E. Marshall and S.~Rohde.
\newblock The {L}oewner differential equation and slit mappings.
\newblock {\em J. Amer. Math. Soc.}, 18(4):763--778, 2005.

\bibitem{Rohde_Schramm}
S.~Rohde and O.~Schramm.
\newblock Basic properties of {SLE}.
\newblock {\em Ann. of Math. (2)}, 161(2):883--924, 2005.

\bibitem{RW}
S.~Rohde and Y.~Wang.
\newblock The {L}oewner energy of loops and regularity of driving functions.
\newblock {\em Int. Math. Res. Not. IMRN}, 2021(10):7715--7763, 2021.

\bibitem{schramm2000scaling}
O.~Schramm.
\newblock Scaling limits of loop-erased random walks and uniform spanning
  trees.
\newblock {\em Israel J. Math.}, 118:221--288, 2000.

\bibitem{TT06}
L.~A. Takhtajan and L.-P. Teo.
\newblock Weil-{P}etersson metric on the universal {T}eichm\"{u}ller space.
\newblock {\em Mem. Amer. Math. Soc.}, 183(861):viii+119, 2006.

\bibitem{W1}
Y.~Wang.
\newblock The energy of a deterministic {L}oewner chain: reversibility and
  interpretation via {${\rm SLE}_{0+}$}.
\newblock {\em J. Eur. Math. Soc. (JEMS)}, 21(7):1915--1941, 2019.

\bibitem{W2}
Y.~Wang.
\newblock Equivalent descriptions of the {L}oewner energy.
\newblock {\em Invent. Math.}, 218(2):573--621, 2019.

\bibitem{W3}
Y.~Wang.
\newblock A note on {L}oewner energy, conformal restriction and {W}erner's
  measure on self-avoiding loops.
\newblock {\em Ann. Inst. Fourier (Grenoble)}, 71(4):1791--1805, 2021.

\bibitem{sle_ld_survey}
Y.~Wang.
\newblock Large deviations of {S}chramm-{L}oewner evolutions: a survey.
\newblock {\em Probab. Surv.}, 19:351--403, 2022.

\bibitem{W_AMS}
Y.~Wang.
\newblock From the random geometry of conformally invariant systems to the
  {K}\"ahler geometry of universal {T}eichm\"uller space, 2024.

\bibitem{WW_St_Flour}
W.~Werner.
\newblock Random planar curves and {S}chramm-{L}oewner evolutions.
\newblock In {\em Lectures on probability theory and statistics}, volume 1840
  of {\em Lecture Notes in Math.}, pages 107--195. Springer, Berlin, 2004.

\bibitem{werner_measure}
W.~Werner.
\newblock The conformally invariant measure on self-avoiding loops.
\newblock {\em J. Amer. Math. Soc.}, 21(1):137--169, 2008.

\end{thebibliography}

\end{document}